\pdfoutput=1
\documentclass[final]{siamart190516}
\usepackage[T1]{fontenc}
\usepackage[latin2]{inputenc}
\usepackage{amssymb}
\usepackage{mathtools}
\usepackage{enumerate}
\usepackage{fnpct}

\newcommand{\NN}{\mathbb{N}}
\newcommand{\ZZ}{\mathbb{Z}}
\newcommand{\CC}{\mathbb{C}}

\DeclareMathOperator*{\argmin}{arg\,min}
\DeclareMathOperator{\fov}{Num}
\DeclareMathOperator{\fovcl}{\overline{Num}}
\DeclareMathOperator{\dist}{dist}
\DeclareMathOperator{\real}{Re}

\newcommand{\LREF}[1]{{\normalfont(\hyperref[#1]{L\labelcref*{#1}})}}
\newcommand{\MREF}[1]{{\normalfont(\hyperref[#1]{M\labelcref*{#1}})}}

\newcommand{\opb}{B}
 
\newsiamremark{remark}{Remark}

\ifpdf
\hypersetup{
  pdftitle={Stability of linear GMRES convergence with respect to compact perturbations},
  pdfauthor={Jan Blechta},
  pdfsubject={MSC2020 65F10},
  pdfkeywords={GMRES, linear convergence, compact perturbation},
}
\fi

\title{Stability of linear GMRES convergence
       with respect to compact perturbations\thanks{%
         Submitted May 26, 2020, revised November 4, 2020
         \funding{This research has been funded by the European Union (EU)
                  -- European Social Fund (ESF) and the Free State of Saxony,
                  project GEOSax (application number 100310486).}}}

\author{Jan Blechta\thanks{%
    Chemnitz University of Technology, Faculty of Mathematics
    (\email{jan.blechta@math.tu-chemnitz.de}).}}

\headers{Linear GMRES convergence and compact perturbations}
        {J.~Blechta}

\begin{document}
\maketitle

\begin{abstract}
  Suppose that a~linear bounded operator $B$ on a~Hilbert space exhibits
  at least
  linear GMRES convergence, i.e., there exists $M_B<1$ such that
  the GMRES residuals fulfill $\|r_k\|\leq M_B\|r_{k-1}\|$
  for every initial residual $r_0$ and step $k\in\NN$.
  We prove that GMRES with a~compactly perturbed operator
  $A=B+C$ admits the bound
  $\|r_k\|/\|r_0\|\leq\prod_{j=1}^k\bigl(M_B+(1+M_B)\,\|A^{-1}\|\,\sigma_j(C)\bigr)$,
  i.e., the singular values $\sigma_j(C)$ control the departure from
  the bound for the unperturbed problem.
  This result can be seen as an extension of
  [{\sc I.~Moret}, {\em A note on the superlinear convergence of {GMRES}},
  SIAM J. Numer. Anal., 34 (1997), pp.~513--516,
  DOI: \href{https://doi.org/10.1137/S0036142993259792}{10.1137/S0036142993259792}],
  where only the case $B=\lambda I$ is considered.
  In this special case $M_B=0$ and the resulting
  convergence is superlinear.
\end{abstract}



\tableofcontents

\section{Introduction}
For a~bounded linear operator $A$ on a~Hilbert space, we consider the equation
$Ax=b$ and its GMRES approximations $x_k$ with residuals $r_k^{A,r_0}:=b-Ax_k$,
i.e., $\|r_k^{A,r_0}\| = \min_{p\in\mathcal{P}_k,\,p(0)=1} \|p(A)r_0\|$.
In this paper we are concerned with convergence of
the GMRES method for $A=B+C$ where $B$ is such that it exhibits
at least
linear GMRES convergence and $C$ is compact. The main result, the proof of which
we postpone to Section~\ref{sec:proof}, is as follows.
\begin{theorem} \label{thm:stab}
  Let $H$ be a~complex separable Hilbert space.
  Suppose $A=B+C$ is invertible with
  $B$ an~invertible bounded linear operator on~$H$ and
  $C$ a~compact linear operator on~$H$ with singular values
  $\|C\|=\sigma_1(C)\geq\sigma_2(C)\geq\cdots\geq0$.
  Further assume that $B$ exhibits linear GMRES convergence, i.e, that
  $B$'s linear reduction factor
  \begin{gather}
    \label{eq:deflinear}
    M_B \coloneqq
    \sup_{z\in H} \sup_{k\in\NN} \frac{\|r_k^{B,z}\|}{\|r_{k-1}^{B,z}\|}
    \intertext{%
      satisfies $M_B < 1$.
      Then for any $r_0\in H$ the GMRES residuals $r_k^{A,r_0}$ fulfill
    }
    \label{eq:stab}
    \frac{\|r_k^{A,r_0}\|}{\|r_0\|}
    \leq
    \prod_{j=1}^k \Bigl( M_B + (1+M_B)\, \|A^{-1}\|\, \sigma_j(C) \Bigr)
    \qquad
    \text{for all } k\in\NN
    \\[-3mm] \intertext{and} \noalign{\vskip-1mm}
    \label{eq:limsup}
    \limsup_{k\rightarrow\infty} \frac{\|r_k^{A,r_0}\|}{\|r_{k-1}^{A,r_0}\|}
    \leq M_B.
    \intertext{Consequently, if there exists $p\geq1$ such that}
    \nonumber
    \|C\|_{\mathcal{S}_p}
    \coloneqq \Bigl( \sum_{j=1}^\infty \sigma_j(C)^p \Bigr)^\frac{1}{p}
    < \infty,
    \\[-4mm] \intertext{then} \noalign{\vskip-2mm}
    \label{eq:stabrate}
    \Biggl( \frac{\|r_k^{A,r_0}\|}{\|r_0\|} \Biggr)^{\frac{1}{k}}
    \leq
    M_B + k^{-\frac{1}{p}}\, (1+M_B)\, \|A^{-1}\|\, \|C\|_{\mathcal{S}_p}
    \qquad
    \text{for all } k\in\NN.
  \end{gather}
\end{theorem}

A~preliminary version of the result in Theorem~\ref{thm:stab}
has been published
in the thesis \cite[Appendix~III.B]{ble2019}, where it has
been used to study the pressure convection--diffusion
preconditioner~\cite{ElmanSilvesterWathen2014}
for the solution of the Navier--Stokes equations.
We expect the result to be applicable to analysis of
a~broad range of problems comming from the numerical solution
of partial differential equations, where compact perturbations
of ``nice'' operators often appear.

The linear reduction factor $M_B$ as given by~\eqref{eq:deflinear}
is the best constant of the estimate
\begin{align}
  \label{eq:deflinear2}
  \|r_k^{B,r_0}\| \leq M_B \|r_{k-1}^{B,r_0}\|
  \qquad \text{for all } k\in\NN \text{ and all } r_0\in H.
\end{align}
Obviously it holds $0\leq M_B\leq1$ but
Theorem~\ref{thm:stab} is non-trivial only when $M_B<1$.
The result can be interpreted as a~type of
stability\footnote{%
  Here the term ``stability'' is \emph{not} used in the sense
  of studying the behavior of GMRES in finite arithmetic precision,
  which itself is a~broad research field
  \cite[section~5.10]{liesen-strakos-2013}.
} of linear GMRES convergence
with respect to compact perturbations. The bound~\eqref{eq:stab}
controls the departure from linear convergence~\eqref{eq:deflinear2}
in terms of the singular values of~$C$. Recall that $\sigma_j(C)\to0$
as $j\to\infty$ for $C$ compact. If additionally $\sum_{j=1}^\infty
\sigma_j(C)^p$ is finite\footnote{%
  Note that $\|\cdot\|_{\mathcal{S}_p}$ in Theorem~\ref{thm:stab}
  is the norm (quasinorm) on the $p$-Schatten--von-Neumann
  ideal~$\mathcal{S}_p(H)$ when $p\geq1$ ($0<p<1$).
  The Hilbert--Schmidt norm $\|\cdot\|_{\mathcal{S}_2}$
  takes the form of the Frobenius matrix norm
  in the finite-dimensional case.
}, then, in view of~\eqref{eq:stabrate},
the mean departure from the linear convergence is
superlinear $\smash{O\bigl(k^{-\frac1p}\bigr)}$.
Furthermore, if $C$ has finite rank $L$, i.e.,
$\sigma_L(C) > \sigma_{L+1}(C) = 0$,
then \eqref{eq:stab} easily implies
\begin{align}
  \label{eq:stabfiniterank}
  \frac{\|r_k^{A,r_0}\|}{\|r_0\|}
  &\leq
  \eta\, M_B^k
  \quad
  \text{for all } k\in\NN,
  &
  \eta
  &\coloneqq
  \prod_{j=1}^L \Bigl( 1 + (1+M_B^{-1})\, \|A^{-1}\|\, \sigma_j(C) \Bigr)
\end{align}
if $M_B>0$; the special case $M_B=0$, which we omit, corresponds to a~finite-rank
perturbation of the identity.

For the sake of illustration we explicitly give a~variant of
Theorem~\ref{thm:stab} for the finite-dimensional case.
\begin{corollary}
  Let $B\in\CC^{N\times N}$ be a~non-singular matrix
  and $C\in\CC^{N\times N}$ be such that $A=B+C$ is non-singular.
  Let $M_B\in[0,1]$ be given by~\eqref{eq:deflinear},
  or equivalently,
  let $M_B\in[0,1]$ be the best possible constant
  from the inequality~\eqref{eq:deflinear2}.
  If $M_B<1$ then \eqref{eq:stab} holds true and \eqref{eq:stabrate},
  in particular, implies that
  \begin{gather*}
    \Biggl( \frac{\|r_k^{A,r_0}\|}{\|r_0\|} \Biggr)^{\frac{1}{k}}
    \leq
    M_B + \frac{1}{\sqrt{k}}\, (1+M_B)\, \|A^{-1}\|\, \|C\|_{\mathcal{S}_2}
    \qquad
    \text{for all } k\in\NN,
  \end{gather*}
  where $\|C\|_{\mathcal{S}_2}$ is the Frobenius matrix norm of~$C$.
\end{corollary}

We will devote Section~\ref{sec:linear} to obtaining an~intrinsic
characterization of~$M_B$ which does not depend on GMRES.
We will show in Propositions~\ref{thm:linear}
and~\ref{lemma:minimizerlambda} that
\begin{align*}
  M_B
  = \min_{\lambda\in\CC} \|I-\lambda B\|
  = \min_{\lambda\in\CC} \|I-\lambda B^{-1}\|
\end{align*}
and that the minima are attained on $\fovcl{B^{-1}}$ and $\fovcl{B}$,
respectively,
where $\fovcl{T}$ stands for the closure of
\begin{gather*}
  \fov T \coloneqq
  \bigl\{ (Tz,z),\, z\in H,\, \|z\|=1 \bigr\},
\end{gather*}
the numerical range (field of values) of~$T$.
The well-known sufficient condition for $M_B<1$ that $0\notin\fovcl{B}$
is due to Elman's bound~\cite{elman1982}
\begin{align*}
  M_B
  &\leq \sqrt{1-\frac{\nu_B^2}{\|B\|^2}}
\end{align*}
and the improved bound by Starke~\cite{starke1997},
Eiermann, and Ernst~\cite{EiermannErnst2001}
\begin{align*}
  M_B
  &\leq \sqrt{1-\nu_B\nu_{B^{-1}}}
  \leq \sqrt{1-\frac{\nu_B^2}{\|B\|^2}},
\end{align*}
where $\nu_{T}=\inf_{\lambda\in\fov{T}}|\lambda|$.
(Note that $\nu_{T^{-1}}>0$ if and only if $\nu_T>0$ for invertible~$T$.)
If $H$ is finite-dimensional, then the linear convergence~\eqref{eq:deflinear2}
with some $M_B<1$ is equivalent to strictly monotone convergence for any
initial residual, i.e.,
\begin{align}
  \label{eq:defmonotone}
  \|r_k^{B,r_0}\| < \|r_{k-1}^{B,r_0}\|
  \quad\text{or}\quad
  \|r_{k-1}^{B,r_0}\| = 0
  \qquad \text{for every } k\in\NN \text{ and every } r_0\in H.
\end{align}
We will also show in Proposition~\ref{thm:linear}
that \eqref{eq:deflinear2} with $M_B<1$ holds
if and only if $0\notin\fovcl{B}$,
that~\eqref{eq:defmonotone} holds
if and only if $0\notin\fov{B}$,
and that the two situations are not, in general, equivalent unless
$H$ is finite-dimensional or $B$ is self-adjoint.
In fact, we will give an~example in Remark~\ref{rem:example}
of a~unitary $B$ with $0\notin\fov{B}$ and
$0\in\fovcl{B}$ so that $B$ exhibits strictly monotone
GMRES convergence~\eqref{eq:defmonotone} but does not exhibit linear
GMRES convergence of any rate, i.e., the best~$M_B$
of~\eqref{eq:deflinear2} is $M_B=1$.

The definition of~$M_B$ that we consider corresponds to the optimal constant
in bounds of the form $ \|r_k^{B,r_0}\| \leq M_B^k\, \|r_0\| $.
On the other hand, one might have a~bound of the type
$ \|r_k^{B,r_0}\| \leq \eta\,M_{B,\eta}^k\, \|r_0\| $
with some $\eta>1$ and $M_{B,\eta}<1$, which can be a~better bound
in the sense that $M_{B,\eta}<M_B$. This type of bounds appears
frequently in the existing literature; see, e.g., \cite{BGT2005,liesen-tichy-2020}.
It is interesting to observe that we have already linked the two types
of bounds by~\eqref{eq:deflinear2} and~\eqref{eq:stabfiniterank}
for the case where $A-B$ is of finite rank. As a~consequence,
the bound $ \|r_k^{A,r_0}\| \leq M_A^k\, \|r_0\| $ with $M_A \leq 1$
can be improved to the bound~\eqref{eq:stabfiniterank} with $M_B < M_A \leq 1$ and $\eta>1$
by considering a~splitting $A=B+C$ with a~finite-rank $C$.
Nevertheless, in the remainder of this paper we focus on
the case without $\eta>1$ and will only shortly revisit
the topic of splittings, albeit in a~different context,
in the closing Section~\ref{sec:outro}.
Note that various splittings $A=B+C$ with $B$ ``good'' in some
sense and $C$ of low rank were
also considered by Huhtanen and Nevanlinna~\cite{huhtanen-nevanlinna-2000}.

Theorem~\ref{thm:stab} can be seen as a~generalization Moret's
result~\cite{moret1997}, which only considers $B=\lambda I$,
$\lambda\neq0$. Indeed, $M_{\lambda I}=0$ and hence
\eqref{eq:limsup} gives the \emph{Q-superlinear convergence}
$\lim_{k\rightarrow\infty} \frac{\|r_k\|}{\|r_{k-1}\|}=0$
of~\cite[Theorem 1]{moret1997} and \eqref{eq:stabrate}
gives the rate $\|r_k\|^\frac1k = O\bigl(k^{-\frac1p}\bigr)$
of~\cite[eq.~(1.1)]{moret1997}; only the estimate
\cite[ineq.~(2.7)]{moret1997} is finer than ours:
\begin{align*}
  \frac{\|r_k^{A,r_0}\|}{\|r_0\|}
  \leq
  \prod_{j=1}^k \sigma_j(A^{-1})\, \sigma_j(C)
  \leq
  \prod_{j=1}^k \|A^{-1}\|\, \sigma_j(C);
\end{align*}
the first inequality is due to Moret and the right-most
expression comes from the bound~\eqref{eq:stab}.

Under the conditions of Theorem~\ref{thm:stab},
Hansmann's theorem \cite[Theorem~2.1]{hansmann2011} provides
a~bound for the accumulation rate of discrete eigenvalues of~$A$
at the numerical range of~$B$. A~straightforward application
of the theorem gives the estimate
\begin{align*}
  \sum_{\lambda\in\sigma_\mathrm{disc}(A)}
  \dist\bigl(\lambda,\fovcl{B}\bigr)^p
  \leq
  \|C\|_{\mathcal{S}_p}^p
  \qquad
  \text{for any } p>1,
\end{align*}
where $\sigma_\mathrm{disc}(A)$ is the set of eigenvalues of~$A$
with finite algebraic multiplicity and each eigenvalue is counted
in the sum according to its
algebraic multiplicity. This estimate can be understood as the counterpart
to the perturbation result of Theorem~\ref{thm:stab} concerning
the perturbation of spectra\footnote{%
  Here, recall that
  concerning the finite-dimensional case
  the spectrum itself is not sufficient for obtaining any useful information
  about the behavior of GMRES \cite{greenbaum-strakos,greenbaum-ptak-strakos}.
  On the other hand, in the infinite-dimensional case the spectrum determines,
  in certain sense, the behavior of GMRES at high iteration counts $k\to\infty$;
  see Section~\ref{sec:outro} and \cite{nevanlinna1993}.
  The relationship between the finite-dimensional
  and infinite-dimensional cases is not, to our best knowledge, well established.
}.

Nevanlinna~\cite{nevanlinna1996,nevanlinna-2003,nevanlinna1993},
Malinen~\cite{malinen1996},
and Hyv\"{o}nen and Nevanlinna~\cite{hyvonen-nevanlinna-2000}
deal with GMRES for compactly perturbed operators
using techniques of complex analysis, which we completely
avoid in this paper.
On the other hand, Huhtanen and Nevanlinna~\cite{huhtanen-nevanlinna-2000} avoid
complex analysis, consider splittings
$A=B+C$ with $B$ normal and $C$ compact (or low-rank),
and obtain lower bounds on GMRES convergence speed.

The outline of the paper is as follows. In Section~\ref{sec:prelim}
we fix the notation and gather some basic facts. Section~\ref{sec:linear}
contains a~characterization of linear convergence and the
reduction factor~$M_B$. In Section~\ref{sec:proof},
building on Moret's result~\cite{moret1997},
we prove Theorem~\ref{thm:stab}. We briefly conclude
and mention interesting questions this research gave rise to
in Section~\ref{sec:outro}.

\section{Preliminaries}
\label{sec:prelim}
Throughout the paper we assume that $H$ is
a~complex Hilbert space with an~inner
product~$(\cdot,\cdot)$ which is linear and antilinear
in the first and second argument, respectively,
and the symbol ${\|\cdot\|}$ stands for
either the norm on~$H$ induced by $(\cdot,\cdot)$ or
the induced operator norm on~$\mathcal{L}(H)$, the space
of bounded linear operators on~$H$.
For operator $A\in\mathcal{L}(H)$, initial residual $r_0\in H$,
and integer $k=0,1,2,\ldots$ we define the Krylov subspace
$\mathcal{K}_k(A,r_0)\coloneqq
\operatorname{span}\{r_0,Ar_0,A^2r_0,\ldots,A^{k-1}r_0\}
\subset H$.
For a~right-hand side $b\in H$ and an~initial guess $x_0\in H$,
the initial residual is given by $r_0=b-Ax_0$.
Then the GMRES algorithm constructs a~sequence
$\{x_k\}_{k=1}^\infty\subset H$ given by minimizing the norm of
residuals $r_k^{A,r_0}=b-Ax_k$ over $x_k\in x_0+\mathcal{K}_k(A,r_0)$, i.e.,
\begin{align}
  \label{eq:gmreskrylov}
  x_k &= \argmin_{\substack{x_k\in H \\x_k-x_0\in\mathcal{K}_k(A,r_0)}} \|b-Ax_k\|.
\end{align}
We will write just $r_k$ instead of~$r_k^{A,r_0}$
if there is no risk of confusion.

\begingroup\setlength\emergencystretch{\hsize}\hbadness=10000
Assume that $t_1,t_2,\ldots,t_k$ is
the ascending orthonormal basis of the spaces $\mathcal{K}_k(A,r_0)$, $k=1,2,\ldots$, and
$z_1,z_2,\ldots,z_k$ is the ascending orthornormal basis of the spaces $A\mathcal{K}_k(A,r_0)$,
$k=1,2,\ldots$. This is well-defined~if
\begin{align*}
  \mathcal{K}_{k+1}(A,r_0) \supsetneq \mathcal{K}_k(A,r_0)
  \qquad \text{for all } k=1,2,\ldots.
\end{align*}
\endgroup
It is well-known that in the converse case, when
$\mathcal{K}_{m+1}(A,r_0)={}\allowbreak \mathcal{K}_{m}(A,r_0)
\supsetneq{}\allowbreak \mathcal{K}_{m-1}(A,r_0)\supsetneq{}\allowbreak \ldots$
for certain $m$, the solution has been reached, i.e., $Ax_m = b$,
provided that $A$ is invertible.
To see this, observe that $A\mathcal{K}_{m}(A,r_0)\subset\mathcal{K}_{m+1}(A,r_0)$
but at the same time
$\dim A\mathcal{K}_{m}(A,r_0)=\dim\mathcal{K}_{m}(A,r_0)=\dim\mathcal{K}_{m+1}(A,r_0)$
by the invertibility of $A$; hence
$A\mathcal{K}_{m}(A,r_0)=\mathcal{K}_{m+1}(A,r_0)\ni r_0$, i.e.,
$r_0=\hat p(A)Ar_0$ with some $\hat p\in\mathcal{P}_{m-1}$
which means that $r_m=r_0-\hat p(A)Ar_0=0$; here and throughout the paper
$\mathcal{P}_n$ stands for the space of polynomials of degree at most~$n$.

Hence, if $A$ is invertible, we can assume that $\{t_j\}_{j=1}^\infty$ and
$\{z_j\}_{j=1}^\infty$ are extended orthonormal systems, i.e.,
either (i) $\{t_j\}_{j=1}^\infty$ and $\{z_j\}_{j=1}^\infty$ are
orthonormal systems, or (ii) $\{t_j\}_{j=1}^\infty$ and $\{z_j\}_{j=1}^\infty$
are such that $\{t_j\}_{j=1}^m$ and $\{z_j\}_{j=1}^m$
are orthonormal systems and $\{t_j\}_{j=m+1}^\infty$ and $\{z_j\}_{j=m+1}^\infty$
are zero sequences. In both cases $t_1,t_2,\ldots,t_k$ and
$z_1,z_2,\ldots,z_k$ are orthonormal bases of $\mathcal{K}_k(A,r_0)$
and $A\mathcal{K}_k(A,r_0)$, respectively, for each $k\in\NN$.

Recall that the numerical range of a~bounded linear operator~$A$ on~$H$ is
a~subset of~$\CC$ given by
\begin{gather*}
  \fov A \coloneqq
  \bigl\{ (Az,z),\, z\in H,\, \|z\|=1 \bigr\}.
  \intertext{The distance of $\fov A$ to the origin is denoted}
  \nu_A \coloneqq
  \inf_{\lambda\in\fov A} |\lambda|.
\end{gather*}
The closure of the numerical range, $\fovcl A$, contains the spectrum
of~$A$. The Toeplitz--Hausdorff theorem
\cite[Theorem 9.3.1]{davies2007}, \cite{gustafson-1970}
says that $\fov A$ is a~convex set. The Lax--Milgram theorem guarantees that
$A$ is invertible provided $0\notin\fovcl A$ and in such a~situation
it holds that $\|A^{-1}\|\leq\nu_A^{-1}$.

\section{Characterization of linear GMRES convergence}
\label{sec:linear}
\begin{proposition}
  \label{thm:linear}
  Let $\opb$ be a~bounded linear operator on a~complex Hilbert space $H$.
  Further suppose that $\opb$ is invertible.
  The linear reduction factor~\eqref{eq:deflinear} fulfills
  \begin{align}
     \label{eq:bndM}
     M_\opb
     = \inf_{\lambda\in\CC} \|I-\lambda \opb\|
     = \inf_{\lambda\in\CC} \|I-\lambda \opb^{-1}\|
     \leq \sqrt{1 - \nu_\opb \nu_{\opb^{-1}}}
     \leq \sqrt{1 - \frac{\nu_\opb^2}{\|\opb\|^2}}
     .
  \end{align}
  Furthermore, the following assertions are equivalent:
  \begin{enumerate}[\normalfont(L1)]
    \item
      \label{L1}
      the operator $\opb$ exhibits
      linear GMRES convergence~\eqref{eq:deflinear2} with
      some $M_\opb<1$;
    \item
      \label{L2}
      $\nu_\opb > 0$;
    \item
      \label{L3}
      $\nu_{\opb^{-1}} > 0$.
  \end{enumerate}
  The following assertions are equivalent as well:
  \begin{enumerate}[\normalfont(M1)]
    \item
      \label{M1}
      the operator $\opb$ exhibits
      strictly monotone GMRES convergence~\eqref{eq:defmonotone};
    \item
      \label{M2}
      $0\notin\fov \opb$;
    \item
      \label{M3}
      $0\notin\fov \opb^{-1}$.
  \end{enumerate}
  If $H$ is finite-dimensional or $\opb$ is self-adjoint then
  \LREF{L1}, \LREF{L2}, \LREF{L3}, \MREF{M1}, \MREF{M2}, and \MREF{M3}
  are equivalent.
\end{proposition}
\begin{proof}
  From~\eqref{eq:gmreskrylov} we have, for all $\lambda\in\CC$,
  $r_0\in H$, and $k\in\NN$,
  \begin{align*}
    \|r_k\|
    = \min_{\substack{p\in\mathcal{P}_k \\ p(0)=1}} \|p(\opb)r_0\|
    \leq \|I-\lambda \opb\|
    \min_{\substack{p\in\mathcal{P}_{k-1} \\ p(0)=1}} \|p(\opb)r_0\|
    = \|I-\lambda \opb\|\, \|r_{k-1}\|.
  \end{align*}
  Dividing by $\|r_{k-1}\|$,
  taking the infimum over $\lambda\in\CC$ and the supremum
  over $r_0\in H$ and $k\in\NN$, we immediately obtain
  $M_\opb\leq\inf_{\lambda\in\CC}\|I-\lambda \opb\|$.
  We continue by showing the opposite inequality.
  The minimax theorem of Asplund and Pt\'{a}k~\cite{asplund-ptak-1971}
  says that\footnote{%
    The equality~\eqref{eq:minimax} expresses the ``equivalence
    of true GMRES and ideal GMRES in the first step.'' This topic
    has received considerable attention in the GMRES literature; see, e.g.,
    \cite{greenbaum-gurvits-1994,greenbaum-trefethen-1994,joubert-1994}.
    For our purpose, in the context of operators on a~Hilbert space,
    we refer to~\cite{asplund-ptak-1971},
    where \eqref{eq:minimax} has been demonstrated to hold true
    if and only if $H$ is an~inner-product space,
    possibly infinite-dimensional.
  }
  \begin{align}
    \label{eq:minimax}
    \inf_{\lambda\in\CC} \|I-\lambda \opb\|
    = \inf_{\lambda\in\CC} \sup_{\substack{z\in H \\ \|z\|=1}} \|z-\lambda \opb z\|
    = \sup_{\substack{z\in H \\ \|z\|=1}} \inf_{\lambda\in\CC} \|z-\lambda \opb z\|
    .
  \end{align}
  The minimization on the right hand-side is nothing other than
  the first GMRES step so that
  $\inf_{\lambda\in\CC} \|I-\lambda \opb\|=
  \sup_{r_0\in H} \frac{\|r_1\|}{\|r_0\|}\leq M_\opb$
  and the first equality in~\eqref{eq:bndM} is thus proved.
  Furthermore, the minimization on the right-hand side of~\eqref{eq:minimax}
  can be carried out explicitly:
  for any $z\in H$, $z\neq0$ it holds that
  \begin{align}
    \label{eq:stepone}
    \min_{\lambda\in\CC} \frac{\|z-\lambda \opb z\|^2}{\|z\|^2}
    = 1 - \frac{|(\opb z,z)|^2}{\|\opb z\|^2 \|z\|^2}
  \end{align}
  and the minimum is attained for $\lambda=\frac{(z,\opb z)}{\|\opb z\|^2}$.
  Equations~\eqref{eq:minimax} and~\eqref{eq:stepone} imply
  \begin{subequations}
    \label{eq:stepone2}
    \begin{align}
      \label{eq:stepone2a}
      \inf_{\lambda\in\CC} \|I-\lambda \opb\|
      &= \sup_{z\in H}
         \sqrt{1-\frac{|(\opb z,z)|^2}{\|\opb z\|^2 \|z\|^2}},
      \\
      \label{eq:stepone2b}
      \inf_{\lambda\in\CC} \|I-\lambda \opb^{-1}\|
      &= \sup_{z\in H}
         \sqrt{1-\frac{|(\opb^{-1}z,z)|^2}{\|\opb^{-1}z\|^2 \|z\|^2}},
    \end{align}
  \end{subequations}
  but by virtue of the invertibility of~$\opb$,
  the right-hand sides in~\eqref{eq:stepone2} are the same, which shows
  the second equality in~\eqref{eq:bndM}.
  Using the inequalities
  \begin{subequations}
    \label{eq:ineq}
    \begin{gather}
      \label{eq:ineqa}
      \inf_{z\in H} \frac{|(\opb z,z)|^2}{\|\opb z\|^2 \|z\|^2}
      \geq \inf_{z\in H} \frac{|(\opb z,z)|}{\|z\|^2}
           \inf_{z\in H} \frac{|(\opb z,z)|}{\|\opb z\|^2}
      = \nu_\opb \nu_{\opb^{-1}},
      \\
      \label{eq:ineqb}
      \nu_{\opb^{-1}}
      = \inf_{z\in H} \frac{|(\opb^{-1}\opb z,\opb z)|}{\|\opb z\|^2}
      \geq \inf_{z\in H} \frac{|(z,\opb z)|}{\|\opb\|^2 \|z\|^2}
      = \frac{\nu_\opb}{\|\opb\|^2},
      \\
      \label{eq:ineqc}
      \nu_\opb
      = \inf_{z\in H} \frac{|(\opb\opb^{-1} z,\opb^{-1} z)|}{\|\opb^{-1} z\|^2}
      \geq \inf_{z\in H} \frac{|(z,\opb^{-1} z)|}{\|\opb^{-1}\|^2 \|z\|^2}
      = \frac{\nu_{\opb^{-1}}}{\|\opb^{-1}\|^2},
    \end{gather}
  \end{subequations}
  together with~\eqref{eq:stepone2}
  immediately yields both inequalities in~\eqref{eq:bndM}.
  Thus~\eqref{eq:bndM} is proved.

  The inequalities \eqref{eq:ineqb}, \eqref{eq:ineqc} establish the equivalence
  of~\LREF{L2} and~\LREF{L3}. The assertion~\LREF{L2}
  or~\LREF{L3} implies~\LREF{L1} by~\eqref{eq:bndM}.
  We proceed by showing that~\LREF{L1} implies~\LREF{L2}.
  Assume that~\LREF{L2} is violated, i.e., that $\nu_\opb=0$.
  Hence there exists $\{z_j\}_{j=1}^\infty\subset H$ such that
  $\|z_j\|=1$ and $(\opb z_j,z_j)\rightarrow0$ as $j\rightarrow\infty$.
  Therefore $\frac{|(\opb z_j,z_j)|}{\|\opb z_j\|} \leq \|\opb^{-1}\|\, |(\opb z_j,z_j)|
  \rightarrow0$. Thus by~\eqref{eq:stepone} we obtain
  \begin{gather*}
    \inf_{\lambda\in\CC} \|z_j-\lambda \opb z_j\|^2
    = 1 - \frac{|(\opb z_j,z_j)|^2}{\|\opb z_j\|^2}
    \rightarrow 1.
    \shortintertext{Thus}
    \sup_{\substack{r_0\in H \\ \|r_0\|=1}}
    \inf_{\lambda\in\CC} \|r_0-\lambda \opb r_0\|
    = 1
  \end{gather*}
  so that $M_\opb=1$ and implication \LREF{L1}$\Rightarrow$\LREF{L2}
  is proved.

  The equivalence of~\MREF{M2} and~\MREF{M3} follows directly
  from the definition of numerical range owing to the invertibility
  of~$\opb$. We proceed by showing that~\MREF{M2} implies~\MREF{M1}.
  We fix $k\in\NN$. If $r_{k-1}=0$ we are done so let us assume
  the converse. We have
  \begin{align*}
    \frac{\|r_k\|}{\|r_{k-1}\|}
    = \min_{\substack{p\in\mathcal{P}_k \\ p(0)=1}}
      \frac{\|p(\opb)r_0\|}{\|r_{k-1}\|}
    \leq \frac{\|(I-\lambda \opb)r_{k-1}\|}{\|r_{k-1}\|}
  \end{align*}
  for any $\lambda\in\CC$. Hence with $\lambda$ minimizing
  the right-hand side we obtain, by~\eqref{eq:stepone}
  and~\MREF{M2},
  \begin{align*}
    \frac{\|r_k\|}{\|r_{k-1}\|} \leq
    \sqrt{1-\frac{|(\opb r_{k-1},r_{k-1})|^2}{\|\opb r_{k-1}\|^2 \|r_{k-1}\|^2}}
    < 1
  \end{align*}
  so that~\MREF{M1} follows.
  On the other hand, if~\MREF{M2} is violated
  then there exists
  $r_0\in H$ with $r_0\neq0$ and $(\opb r_0,r_0)=0$
  so that~\eqref{eq:stepone} immediately implies
  $\|r_1\|=\|r_0\|$, which contradicts~\MREF{M1}.
  Hence the equivalence of~\MREF{M1} and~\MREF{M2}
  is now clear.

  If $H$ is finite-dimensional then $\fov \opb$
  is closed and thus \LREF{L2} and \MREF{M2} are equivalent.
  Now assume $\opb$ is self-adjoint. Clearly \LREF{L2} implies
  \MREF{M2}. To prove the opposite, assume that $\nu_\opb=0$
  but $0\notin\fov \opb$. As $\fov \opb$ is real and convex, it must be
  of the form $[a,0)$, $(a,0)$, $(0,b)$, or $(0,b]$ with
  some $a<0$ or $b>0$. On the other hand, the endpoints of
  $\fov \opb$ belong to the spectrum of~$\opb$
  \cite[Satz~1]{hildebrandt-1966},
  which is always
  closed. Hence the conclusion is that $0$ is in the spectrum,
  which contradicts the invertibility of~$\opb$.
\end{proof}

\begin{remark}
  \label{rem:example}
  We give an~example of a~unitary operator fulfilling~\MREF{M1}
  but contradicting~\LREF{L1}, thus showing that~\LREF{L1}
  and~\MREF{M1} are not in general equivalent.
  Let $\opb$ be a~diagonal operator on~$\ell^2$
  given by
  \begin{align*}
    \opb = \sum_{j\in\ZZ} \lambda_j e_j (\cdot,e_j),
    \qquad
    \lambda_j=e^{i\arctan j} \text{ for }j\in\ZZ,
  \end{align*}
  where $\{e_j\}_{j\in\ZZ}$ is the canonical basis in~$\ell^2$,
  which is orthonormal with respect to $(\cdot,\cdot)$.
  The eigenvalues $\{\lambda_j\}_{j\in\ZZ}$ fulfill
  $|\lambda_j|=1$ and $\real\lambda_j>0$ and,
  indeed, $\opb$ is unitary and hence invertible. In fact,
  the spectrum of~$\opb$ is $\{\lambda_j\}_{j\in\ZZ} \cup \{i,-i\}$.
  Indeed, $\|(\opb\mp iI)e_k\|=|\lambda_k\mp i|\rightarrow0$
  as $k\rightarrow\pm\infty$, i.e., $\pm i$ is in
  the approximate point spectrum of~$\opb$.
  For $f_k\coloneqq e_k+e_{-k}$ it holds that
  $\frac{(\opb f_k,f_k)}{\|f_k\|^2}=\real\lambda_k\rightarrow0$
  as $k\rightarrow\pm\infty$. Hence $\nu_\opb=0$, and thus
  \LREF{L2}, and in turn \LREF{L1}, do not hold true.
  On the other hand, $(\opb z,z)=\sum_{j\in\ZZ}\lambda_j|(z,e_j)|^2$
  which is non-zero whenever $z\neq0$, because
  $\real\lambda_j>0$ for all $j\in\ZZ$. Thus $0\notin\fov \opb$,
  which is~\MREF{M2} so that~\MREF{M1} holds true.
\end{remark}

The following result characterizes $\lambda$ that minimize~\eqref{eq:bndM}.
\begin{proposition}
  \label{lemma:minimizerlambda}
  Let $H$ be a~complex Hilbert space and $\opb\in\mathcal{L}(H)$ be invertible.
  There exist $\lambda_\opb\in\CC$ and $\lambda_{\opb^{-1}}\in\CC$ such that the infima
  on the left-hand sides of~\eqref{eq:stepone2a} and~\eqref{eq:stepone2b},
  respectively, are attained.
  Any such $\lambda_\opb$ and $\lambda_{\opb^{-1}}$ fulfill
  \begin{align*}
    \lambda_\opb
    &\in\fovcl \opb^{-1},
    &
    \|\lambda_\opb \opb\|
    &\leq 1 + M_\opb,
    \\
    \lambda_{\opb^{-1}}
    &\in\fovcl \opb,
    &
    \|\lambda_{\opb^{-1}} \opb^{-1}\|
    &\leq 1 + M_\opb.
  \end{align*}
\end{proposition}
\begin{proof}
  We will first prove that any minimizer of~\eqref{eq:stepone2a}, if it
  exists, fulfills $\lambda_\opb\in\fovcl{\opb^{-1}}$. For the sake of contradiction
  assume that $\lambda_\opb\in\CC\setminus\fovcl{\opb^{-1}}$. Let $z\in H$ with
  $\|z\|=1$ be arbitrary. The function $\lambda\mapsto\|z-\lambda \opb z\|^2$
  is minimized by $\lambda_z = \frac{(z,\opb z)}{\|\opb z\|^2}\in\fov{\opb^{-1}}$;
  see \eqref{eq:stepone}. Hence $|\lambda_\opb-\lambda_z|\geq C>0$,
  where $C=\dist(\lambda_\opb,\fov{\opb^{-1}})$ is independent of~$z$.
  Now consider that
  \begin{align*}
    \|z-\lambda_\opb \opb z\|^2 &=
    1 - 2\real\bigl( \lambda_\opb (\opb z,z) \bigr) + |\lambda_\opb|^2\,\|\opb z\|^2,
    \\
    \|z-\lambda_z \opb z\|^2 &=
    1 - \frac{|(\opb z,z)|^2}{\|\opb z\|^2},
  \end{align*}
  where the second equality has been already shown in~\eqref{eq:stepone}.
  Combining the last two equalities we obtain
  \begin{gather*}
    \begin{aligned}
      \|z-\lambda_\opb \opb z\|^2 - \|z-\lambda_z \opb z\|^2
      &= \frac{|(\opb z,z)|^2}{\|\opb z\|^2}
      - 2\real\bigl( \lambda_\opb (\opb z,z) \bigr) + |\lambda_\opb|^2\,\|\opb z\|^2
      \\
      &= \|\opb z\|^2\, |\lambda_\opb-\lambda_z|^2
      \geq \|\opb^{-1}\|^{-2}\, C^2.
    \end{aligned}
    \shortintertext{Hence}
    \begin{aligned}
      \|I-\lambda_\opb \opb\|^2 = \sup_{\|z\|=1} \|z-\lambda_\opb \opb z\|^2
      &\geq \sup_{\|z\|=1} \|z-\lambda_z \opb z\|^2 + \|\opb^{-1}\|^{-2}\, C^2
      \\
      &> \sup_{\|z\|=1} \|z-\lambda_z \opb z\|^2
      = \inf_{\lambda\in\CC} \|I-\lambda \opb\|^2,
    \end{aligned}
  \end{gather*}
  where we used~\eqref{eq:minimax} in the last equality,
  and this gives the desired contradiction.
  Hence any minimizers of~\eqref{eq:stepone2a} belong
  to the compact set $\fovcl{\opb^{-1}}$, where  the function
  to be minimized is continuous, so the existence is also proved.
  By the triangle inequality we have
  $\|\lambda_\opb \opb\| \leq \|I\| + \|I-\lambda_\opb \opb\| = 1 + M_\opb$.
  The proof for $\lambda_{\opb^{-1}}$ is analogous.
\end{proof}

\section{Proof of the main result}
\label{sec:proof}
Moret~\cite[Lemma 6]{moret1997}, as well as
Eiermann and Ernst~\cite[Lemma 3.3]{EiermannErnst2001},
proves the following auxiliary result,
which will be crucial for us.
\begingroup\setlength\emergencystretch{\hsize}\hbadness=10000
\begin{lemma}[Moret's formula]
  \label{lemma:moret}
  Let $H$ be a~separable complex Hilbert space,
  $A\in\mathcal{L}(H)$ be invertible,
  and $r_0\in H$ be given.
  For every $k\in\NN$ and $\lambda\in\CC$ the GMRES residuals
  $r_k \coloneqq r_k^{A,r_0}$ fulfill
  \begin{align}
    \label{eq:moret}
    \|r_k\|
    = |(t_{k+1},z_k)| \|r_{k-1}\|
    = |(t_{k+1},(I-\lambda A^{-1})z_k)| \|r_{k-1}\|,
  \end{align}
  where $\{t_j\}_{j=1}^\infty$ and $\{z_j\}_{j=1}^\infty$
  are the ascending extended orthonormal bases of
  $\{\mathcal{K}_1(A,r_0),\allowbreak \mathcal{K}_2(A,r_0),\allowbreak \ldots\}$
  and
  $\{A\mathcal{K}_1(A,r_0),\allowbreak A\mathcal{K}_2(A,r_0),\allowbreak \ldots\}$,
  respectively.
\end{lemma}
\endgroup
\noindent
Note that the second equality in~\eqref{eq:moret} follows trivially
from the definition of $t_{k+1}$ and $z_k$;
indeed $A^{-1}z_k\in\mathcal{K}_k(A,r_0)$, which is a~space spanned
by $\{t_1,\ldots,t_k\}$, and $(t_{k+1},t_j)=0$ for all $j\leq k$.

The approximation numbers of a~bounded linear operator $T$
on a~Hilbert space~$H$ are defined by
\begin{align}
  \label{eq:apnums}
  \sigma_j(T) = \inf_{\substack{
      M\in\mathcal{L}(H) \\
      \operatorname{rank}(M) < j
    }} \|T-M\|,
    \qquad j=1,2,\ldots
\end{align}
The approximation numbers satisfy
\begin{subequations}
  \label{eq:snum}
  \begin{gather}
    \label{eq:snum1}
    \|T\| = \sigma_1(T) \geq \sigma_2(T) \geq \ldots \geq 0,
    \qquad \qquad \qquad \qquad \quad
    \\
    \label{eq:snum2}
    \sigma_{j+k-1}(S+T) \leq \sigma_j(S) + \sigma_k(T),
    \qquad
    j,k=1,2,\ldots,
    \\
    \label{eq:snum3}
    \sigma_j(WTU) \leq \|W\|\, \sigma_j(T)\, \|U\|,
    \qquad \quad
    j=1,2,\ldots
  \end{gather}
\end{subequations}
with any $S,T,U,W\in\mathcal{L}(H)$;
see~\cite[paragraph~2.2.1, p.~79, Theorem~2.3.3, p.~83]{pietsch1987}.
If $T$ is compact, the numbers $\sigma_j(T)$ are the singular
values of~$T$.

\begin{lemma}[{Pietsch~\cite[Lemma 2.11.13, p. 125]{pietsch1987}}]
  Let $T$ be a~bounded linear operator on a~complex Hilbert space $H$.
  Let $\sigma_1(T) \geq \sigma_2(T) \geq \sigma_3(T) \geq \ldots \geq 0$
  denote the approximation numbers of $T$ as defined by~\eqref{eq:apnums}.
  Then for any pair of orthonormal families
  $\{f_1,f_2,\ldots,f_k\}$, $\{g_1,g_2,\ldots,g_k\}\subset H$
  it holds that
  \begin{align}
    \label{eq:determinant}
    \det\bigl\{ (Tf_i,g_j) \bigr\}_{i,j=1}^{k} \leq \prod_{j=1}^k \sigma_k(T).
  \end{align}
\end{lemma}
Moret \cite[ineq.~(2.7)]{moret1997} proves that if $A-\lambda I$
is compact for some $\lambda\in\CC$, then
\begin{align*}
  \frac{\|r_k\|}{\|r_0\|} \leq \prod_{j=1}^k
  \sigma_j(A-\lambda I)\,\sigma_j(A^{-1}).
\end{align*}
We will need a~modification, which follows.
\begin{lemma}
  Suppose that the assumptions of Lemma~\ref{lemma:moret} are fullfilled.
  Then for every $k\in\NN$ and $\lambda\in\CC$ the GMRES residuals
  $r_k \coloneqq r_k^{A,r_0}$ fulfill
  \begin{align}
    \label{eq:resestapprox}
    \frac{\|r_k\|}{\|r_0\|} \leq \prod_{j=1}^k \sigma_j(I-\lambda A^{-1}),
  \end{align}
  where
  $\sigma_1(I-\lambda A^{-1}) \geq \sigma_2(I-\lambda A^{-1}) \geq \ldots \geq 0$
  are the approximation numbers of $I - \lambda A^{-1}$
  as defined by~\eqref{eq:apnums}.
\end{lemma}
\begin{proof}
  From~\eqref{eq:moret} we have
  \begin{gather*}
    \frac{\|r_k\|}{\|r_0\|}
    = \prod_{j=1}^k |(t_{j+1},(I-\lambda A^{-1})z_j)|.
    \shortintertext{%
      The matrix
    }
    \bigl\{ |(t_{i+1},(I-\lambda A^{-1})z_j)| \bigr\}_{i,j=1}^{k}
  \end{gather*}
  is upper triangular because by construction,
  \begin{align*}
    0 &= (t_{j+2},z_j) = (t_{j+3},z_j) = \ldots, \\
    0 &= (t_{j+1},A^{-1}z_j) = (t_{j+2},A^{-1}z_j) = \ldots.
  \end{align*}
  This implies that $\prod_{j=1}^k |(t_{j+1},(I-\lambda A^{-1})z_j)| =
  \det\bigl\{ |(t_{i+1},(I-\lambda A^{-1})z_j)| \bigr\}_{i,j=1}^{k}$
  which is bounded by
  $\prod_{j=1}^k \sigma_j(I-\lambda A^{-1})$ due to~\eqref{eq:determinant}.
\end{proof}
\begin{proof}[Proof of Theorem~\ref{thm:stab}]
  By virtue of the invertibility of $A=B+C$ and the invertibility
  of~$B$ we have
  $A^{-1}=B^{-1}-B^{-1}CA^{-1}$.
  Hence by Lemma~\ref{lemma:moret} we have, for any $\lambda\in\CC$,
  \begin{align*}
    \frac{\|r_k^{A,r_0}\|}{\|r_{k-1}^{A,r_0}\|}
    &\leq |(t_{k+1},(I-\lambda B^{-1})z_k)|
        + |(t_{k+1},\lambda B^{-1}CA^{-1}z_k)|
    \\
    &\leq \|I-\lambda B^{-1}\| + \|\lambda B^{-1}CA^{-1}z_k\|.
  \end{align*}
  The sequence $\{z_k\}_{k=1}^\infty$ is an extended orthonormal system, so that
  Bessel's inequality
  \begin{align*}
    \sum_{k=1}^\infty |(y, z_k)|^2 \leq \|y\|^2
    \qquad \text{for all } y \in H
  \end{align*}
  immediately implies that $(y, z_k)\to0$ as $k\to\infty$
  for every $y\in H$, i.e.,
  $\{z_k\}_{k=1}^\infty$ is weakly null. The compactness
  of $C$ implies the strong convergence $\|CA^{-1}z_k\|\to0$
  and consequently
  \begin{align*}
    \limsup_{k\to\infty}
    \frac{\|r_k^{A,r_0}\|}{\|r_{k-1}^{A,r_0}\|}
    &\leq \|I-\lambda B^{-1}\|
    \qquad \text{for any } \lambda\in\CC.
  \end{align*}
  By minimizing the right-hand side in $\lambda$
  and using the equality
  in~\eqref{eq:bndM}
  we obtain~\eqref{eq:limsup}.

  Using~\eqref{eq:resestapprox} and~\eqref{eq:snum}
  we obtain, for any $\lambda\in\CC$,
  \begin{align*}
    \frac{\|r_k^{A,r_0}\|}{\|r_0\|}
    \leq \prod_{j=1}^k \sigma_j(I-\lambda A^{-1})
    &\leq \prod_{j=1}^k \Bigl(
        \sigma_1(I-\lambda B^{-1}) + \sigma_j(\lambda B^{-1}CA^{-1})
    \Bigr)
    \\
    &\leq \prod_{j=1}^k \Bigl(
        \|I-\lambda B^{-1}\| + \|\lambda B^{-1}\|\,\|A^{-1}\|\,\sigma_j(C)
    \Bigr).
  \end{align*}
  Taking $\lambda\in\CC$ minimizing $\|I-\lambda B^{-1}\|$, we
  obtain~\eqref{eq:stab} using Proposition~\ref{lemma:minimizerlambda}.

  From~\eqref{eq:stab} we obtain,
  using the inequality of arithmetic and geometric means
  and H\"older's inequality,
  \begin{align*}
    \biggl( \frac{\|r_k^{A,r_0}\|}{\|r_0\|} \biggr)^\frac1k
    &\leq \tfrac1k \sum_{j=1}^k \Bigl( M_B + (1+M_B)\, \|A^{-1}\|\, \sigma_j(C) \Bigr)
    \\
    &= M_B + \tfrac1k\, (1+M_B)\, \|A^{-1}\|\, \sum_{j=1}^k \sigma_j(C)
    \\
    &\leq M_B + \tfrac1k\, (1+M_B)\, \|A^{-1}\|\,
               \Biggl( \sum_{j=1}^k \sigma_j(C)^p \Biggr)^{\frac1p} k^{\frac{p-1}{p}},
  \end{align*}
  which immediately yields~\eqref{eq:stabrate}.
\end{proof}

\section{Conclusion and outlook}
\label{sec:outro}
We generalized Moret's result~\cite{moret1997} in Theorem~\ref{thm:stab}
and thus obtained a~certain stability of linear GMRES convergence
under compact perturbations. Indeed, the superlinear
convergence of~\cite{moret1997} follows as a~corollary.
In Section~\ref{sec:linear} we obtained a~useful characterization
of linear convergence, which sheds some light on the limits of
the applicability of Theorem~\ref{thm:stab}.

Bound~\eqref{eq:limsup} has an interesting consequence:
\begin{align}
  \label{eq:limsup2}
  \limsup_{k\rightarrow\infty} \frac{\|r_k^{A,r_0}\|}{\|r_{k-1}^{A,r_0}\|}
  \leq \inf_{K\in\mathcal{C}(H)} M_{A+K}
  \eqqcolon
  M_{\pi(A)},
\end{align}
where the infimum is taken over $\mathcal{C}(H)$, the ideal
of all compact operators in $\mathcal{L}(H)$.
Here $\pi$ is the quotient map
$\pi:\mathcal{L}(H)\to\mathcal{L}(H)/\mathcal{C}(H)$
with the quotient space known as the Calkin algebra
and the quotient norm
$\|\pi(T)\|=\inf_{K\in\mathcal{C}(H)}\|T+K\|$.
In particular, every finite-dimensional eigenspace can be ``removed''
from~$A$ when computing the Q-rate~\eqref{eq:limsup2}. Interestingly,
for the R-rate, one has
\begin{align*}
  \limsup_{k\rightarrow\infty}
  \Biggl( \frac{\|r_k^{A,r_0}\|}{\|r_0\|} \Biggr)^{\frac{1}{k}}
  \leq
  \lim_{k\to\infty}
  \inf_{\substack{p\in\mathcal{P}_k\\p(0)=1}}
  \|p(A)\|^{\frac1k}
  =
  \inf_{\substack{k\in\NN\\p\in\mathcal{P}_k\\p(0)=1}} \max_{z\in\sigma_0(A)}
  |p(z)|^{\frac1k},
\end{align*}
where the equality is due to Nevanlinna~\cite[Theorem~3.3.4]{nevanlinna1993}
and $\sigma_0(A)$ consists of the spectrum of~$A$ except for
its isolated points.
By this analogy, one could conjecture that $M_{\pi(A)}$ also
``cannot see'' any eigenspaces, including those of infinite dimension.
One might wonder whether there is a~useful characterization of~$M_{\pi(A)}$.

Another interesting question concerns $s$-step linear convergence
$\|r_k\|\leq M\|r_{k-s}\|$ with some $s\in\NN$ and $M<1$.
The present result assumes linear reduction in every step, i.e.,
$s=1$. On the other hand, an~important class of problems, e.g.,
self-adjoint operators, exhibit a~reduction in every second step,
i.e., $s=2$ and $M<1$. To this end, we cannot see how the technique
of Section~\ref{sec:proof} generalizes to $s\geq2$. Note that
\cite{huhtanen-nevanlinna-2000}
is concerned with related issues but
does not directly provide such stability we obtained for the case $s=1$.
Also note that pseudospectral techniques have been used to obtain bounds
for the case $s>1$; see, e.g., \cite{embree1999,trefethen-embree-2005}.
 
\section*{Acknowledgement}
The author would like to thank
Erin Carson, Oliver Ernst, Zden\v{e}k Strako\v{s}, and Petr Tich\'{y}
for fruitful discussions, as well as
the two anonymous referees, who provided
useful suggestions which improved the paper.

\bibliographystyle{siamplain}

\begin{thebibliography}{10}

\bibitem{asplund-ptak-1971}
{\sc E.~Asplund and V.~Pt\'{a}k}, {\em A minimax inequality for operators and a
  related numerical range}, Acta Math., 126 (1971), pp.~53--62,
  \url{https://doi.org/10.1007/BF02392025}.

\bibitem{BGT2005}
{\sc B.~Beckermann, S.~A. Goreinov, and E.~E. Tyrtyshnikov}, {\em Some remarks
  on the {E}lman estimate for {GMRES}}, SIAM J. Matrix Anal. Appl., 27 (2005),
  pp.~772--778, \url{https://doi.org/10.1137/040618849}.

\bibitem{ble2019}
{\sc J.~Blechta}, {\em Towards efficient numerical computation of flows of
  non-{N}ewtonian fluids}, PhD thesis, Charles University, Faculty of
  Mathematics and Physics, 2019,
  \url{https://hdl.handle.net/20.500.11956/108384}.

\bibitem{davies2007}
{\sc E.~B. Davies}, {\em Linear Operators and their Spectra}, Cambridge Studies
  in Advanced Mathematics, Cambridge University Press, 2007,
  \url{https://doi.org/10.1017/CBO9780511618864}.

\bibitem{EiermannErnst2001}
{\sc M.~Eiermann and O.~G. Ernst}, {\em Geometric aspects of the theory of
  {K}rylov subspace methods}, Acta Numer., 10 (2001), pp.~251--312,
  \url{https://doi.org/10.1017/S0962492901000046}.

\bibitem{elman1982}
{\sc H.~Elman}, {\em Iterative methods for large, sparse, nonsymmetric systems
  of linear equations}, PhD thesis, Yale University, 1982,
  \url{http://ftp.cs.yale.edu/publications/techreports/tr229.pdf}.
\newblock Research Report \#229.

\bibitem{ElmanSilvesterWathen2014}
{\sc H.~C. Elman, D.~J. Silvester, and A.~J. Wathen}, {\em Finite elements and
  fast iterative solvers: with applications in incompressible fluid dynamics},
  Oxford University Press, 2nd~ed., 2014,
  \url{https://doi.org/10.1093/acprof:oso/9780199678792.001.0001}.

\bibitem{embree1999}
{\sc M.~Embree}, {\em How descriptive are {GMRES} convergence bounds?}, Tech.
  Report~08, Oxford University Computing Laboratory, 1999,
  \url{https://ora.ox.ac.uk/objects/uuid:8ca2d383-4d7d-4e21-805c-98e16537d3d3}.

\bibitem{greenbaum-gurvits-1994}
{\sc A.~Greenbaum and L.~Gurvits}, {\em Max-min properties of matrix factor
  norms}, SIAM J. Sci. Comput., 15 (1994), pp.~348--358,
  \url{https://doi.org/10.1137/0915024}.
\newblock Iterative methods in numerical linear algebra (Copper Mountain
  Resort, CO, 1992).

\bibitem{greenbaum-ptak-strakos}
{\sc A.~Greenbaum, V.~Pt{\'a}k, and Z.~Strako{\v s}}, {\em Any nonincreasing
  convergence curve is possible for {GMRES}}, SIAM J. Matrix Anal. Appl., 17
  (1996), pp.~465--469, \url{https://doi.org/10.1137/S0895479894275030}.

\bibitem{greenbaum-strakos}
{\sc A.~Greenbaum and Z.~Strako\v{s}}, {\em Matrices that generate the same
  {K}rylov residual spaces}, in Recent advances in iterative methods, vol.~60
  of IMA Vol. Math. Appl., Springer, New York, 1994, pp.~95--118,
  \url{https://doi.org/10.1007/978-1-4613-9353-5_7}.

\bibitem{greenbaum-trefethen-1994}
{\sc A.~Greenbaum and L.~N. Trefethen}, {\em G{MRES}/{CR} and
  {A}rnoldi/{L}anczos as matrix approximation problems}, SIAM J. Sci. Comput.,
  15 (1994), pp.~359--368, \url{https://doi.org/10.1137/0915025}.
\newblock Iterative methods in numerical linear algebra (Copper Mountain
  Resort, CO, 1992).

\bibitem{gustafson-1970}
{\sc K.~Gustafson}, {\em The {T}oeplitz-{H}ausdorff theorem for linear
  operators}, Proc. Amer. Math. Soc., 25 (1970), pp.~203--204,
  \url{https://doi.org/10.2307/2036559}.

\bibitem{hansmann2011}
{\sc M.~Hansmann}, {\em An eigenvalue estimate and its application to
  non-selfadjoint {J}acobi and {S}chr\"{o}dinger operators}, Lett. Math. Phys.,
  98 (2011), pp.~79--95, \url{https://doi.org/10.1007/s11005-011-0494-9}.

\bibitem{hildebrandt-1966}
{\sc S.~Hildebrandt}, {\em \"{U}ber den numerischen {W}ertebereich eines
  {O}perators}, Math. Ann., 163 (1966), pp.~230--247,
  \url{https://doi.org/10.1007/BF02052287}.

\bibitem{huhtanen-nevanlinna-2000}
{\sc M.~Huhtanen and O.~Nevanlinna}, {\em Minimal decompositions and iterative
  methods}, Numer. Math., 86 (2000), pp.~257--281,
  \url{https://doi.org/10.1007/PL00005406}.

\bibitem{hyvonen-nevanlinna-2000}
{\sc S.~Hyv\"{o}nen and O.~Nevanlinna}, {\em Robust bounds for {K}rylov
  methods}, BIT, 40 (2000), pp.~267--290,
  \url{https://doi.org/10.1023/A:1022390923774}.

\bibitem{joubert-1994}
{\sc W.~Joubert}, {\em A robust {GMRES}-based adaptive polynomial
  preconditioning algorithm for nonsymmetric linear systems}, SIAM J. Sci.
  Comput., 15 (1994), pp.~427--439, \url{https://doi.org/10.1137/0915029}.
\newblock Iterative methods in numerical linear algebra (Copper Mountain
  Resort, CO, 1992).

\bibitem{liesen-strakos-2013}
{\sc J.~Liesen and Z.~Strako\v{s}}, {\em Krylov subspace methods: Principles
  and analysis}, Numerical Mathematics and Scientific Computation, Oxford
  University Press, Oxford, 2013,
  \url{https://doi.org/10.1093/acprof:oso/9780199655410.001.0001}.

\bibitem{liesen-tichy-2020}
{\sc J.~Liesen and P.~Tich\'{y}}, {\em The field of values bounds on ideal
  {GMRES}}, 2020, \url{https://arxiv.org/abs/1211.5969v3}.

\bibitem{malinen1996}
{\sc J.~Malinen}, {\em On the properties for iteration of a compact operator
  with unstructured perturbation}, Tech. Report A360, Helsinki University of
  Technology, Institute of Mathematics, 1996,
  \url{https://math.aalto.fi/~jmalinen/MyPSFilesInWeb/SmallCompact.pdf}.

\bibitem{moret1997}
{\sc I.~Moret}, {\em A note on the superlinear convergence of {GMRES}}, SIAM J.
  Numer. Anal., 34 (1997), pp.~513--516,
  \url{https://doi.org/10.1137/S0036142993259792}.

\bibitem{nevanlinna1993}
{\sc O.~Nevanlinna}, {\em Convergence of iterations for linear equations},
  Lectures in Mathematics ETH Z\"{u}rich, Birkh\"{a}user Verlag, Basel, 1993,
  \url{https://doi.org/10.1007/978-3-0348-8547-8}.

\bibitem{nevanlinna1996}
{\sc O.~Nevanlinna}, {\em Convergence of {K}rylov methods for sums of two
  operators}, BIT, 36 (1996), pp.~775--785,
  \url{https://doi.org/10.1007/BF01733791}.

\bibitem{nevanlinna-2003}
{\sc O.~Nevanlinna}, {\em Meromorphic functions and linear algebra}, vol.~18 of
  Fields Institute Monographs, American Mathematical Society, Providence, RI,
  2003, \url{https://doi.org/10.1090/fim/018}.

\bibitem{pietsch1987}
{\sc A.~Pietsch}, {\em Eigenvalues and {$s$}-numbers}, vol.~13 of Cambridge
  Studies in Advanced Mathematics, Cambridge University Press, 1987.

\bibitem{starke1997}
{\sc G.~Starke}, {\em Field-of-values analysis of preconditioned iterative
  methods for nonsymmetric elliptic problems}, Numer. Math., 78 (1997),
  pp.~103--117, \url{https://doi.org/10.1007/s002110050306}.

\bibitem{trefethen-embree-2005}
{\sc L.~N. Trefethen and M.~Embree}, {\em Spectra and pseudospectra}, Princeton
  University Press, Princeton, NJ, 2005.
\newblock The behavior of nonnormal matrices and operators.

\end{thebibliography}

\end{document}